\documentclass[a4paper,11pt, reqno]{amsart}

\usepackage{tikz}
\usepackage{subfigure}
\usepackage[final]{microtype}
\usepackage{lmodern}

\usepackage{amsmath, amssymb, amsfonts, amsthm}
\usepackage[latin1]{inputenc}
\usepackage[T1]{fontenc}

\usepackage[draft=false,pdftex]{hyperref}
\hypersetup{
    bookmarks=true,         % show bookmarks bar?
    unicode=false,          % non-Latin characters in Acrobat’s bookmarks
    pdftoolbar=true,        % show Acrobat’s toolbar?
    pdfmenubar=true,        % show Acrobat’s menu?
    pdffitwindow=false,     % window fit to page when opened
    pdfstartview={FitH},    % fits the width of the page to the window
    pdftitle={Bounding the number of edges of matchstick graphs},    % title
    pdfauthor={Jérémy Lavollée and Konrad J. Swanepoel},     % author
    pdfnewwindow=false,      % links in new window
    colorlinks=true,       % false: boxed links; true: colored links
    linkcolor=black,          % color of internal links
    citecolor=black,        % color of links to bibliography
    filecolor=black,      % color of file links
    urlcolor=black           % color of external links
}

\theoremstyle{plain}
\newtheorem{theorem}{Theorem}
\newtheorem{lemma}[theorem]{Lemma}
\newtheorem{corollary}[theorem]{Corollary}
\newtheorem{conjecture}{Conjecture}

\theoremstyle{remark}
\newtheorem{remark}[theorem]{Remark}

\definecolor{konrad}{rgb}{0,0,1}
\definecolor{jeremy}{rgb}{1,0,0}

\renewcommand{\geq}{\geqslant}
\renewcommand{\leq}{\leqslant}

\begin{document}

\title{Bounding the number of edges of matchstick graphs}
\author{J\'er\'emy Lavoll\'ee}
\address[J. Lavoll\'ee, K. J. Swanepoel]{Department of Mathematics, The London School of Economics and Political Science, Houghton Street, London WC2A 2AE}
%\curraddr[J. Lavoll\'ee]{...}

\author{Konrad J. Swanepoel}
%\address{Department of Mathematics, The London School of Economics and Political Science, Houghton Street, London WC2A 2AE}

\date{}
\subjclass[2020]{Primary 52C10. Secondary 05C10}
\keywords{matchstick graph, penny graph, plane unit-distance graph}

\begin{abstract}
We show that a matchstick graph with $n$ vertices has no more than $3n-c\sqrt{n-1/4}$ edges, where $c=\frac12(\sqrt{12} + \sqrt{2\pi\sqrt{3}})$.
The main tools in the proof are the Euler formula, the isoperimetric inequality, and an upper bound for the number of edges in terms of $n$ and the number of non-triangular faces.
We also find a sharp upper bound for the number of triangular faces in a matchstick graph.
\end{abstract}

\maketitle

\section{Introduction}
Matchstick graphs, first introduced by Harborth in 1981 \cite{oberwolfach, lighter}, are graphs drawn in the plane with each edge a straight-line segment of unit length, such that no two edges have a point in common, unless the common point is an endpoint of both edges (Figure~\ref{fig2}).
\begin{figure}[h]
    \centering
    \scalebox{0.7}{\begin{tikzpicture}[line cap=round,line join=round,x=1cm,y=1cm]
%\clip(-20.554372455214974,7.149287698776122) rectangle (-9.167895988616257,13.427664306053453);
\draw[line width=1pt] (-14.08,10.24) -- (-12.0903890116217,10.036411898863616) -- (-12.908882038318504,11.86125960901608) -- cycle;
\draw[line width=1pt] (-10.544833985948046,11.305765932462956) -- (-11.631628826631452,9.626814779432014) -- (-9.63421705605178,9.525098415213794) -- cycle;
\draw[line width=1pt] (-13.177183852305077,8.357460745832702) -- (-11.631628826631426,9.626814779432042) -- (-11.305113499974977,7.65364784745228) -- cycle;
\draw[line width=1pt] (-14.08,10.24) -- (-12.0903890116217,10.036411898863616) -- (-12.908882038318504,11.86125960901608) -- cycle;
\draw[line width=1pt] (-19.531096688913404,10.716135043247027) -- (-17.53197954300168,10.775554206163957) -- (-18.58299662051521,12.477130858206086) -- cycle;
\draw[line width=1pt] (-18.58299662051521,12.477130858206086) -- (-17.53197954300168,10.775554206163957) -- (-16.583879474603485,12.536550021123016) -- cycle;
\draw[line width=1pt] (-17.556084482653166,10.400973520971455) -- (-19.531096688913404,10.716135043247027) -- (-18.816528470369306,8.848143538703523) -- cycle;
\draw[line width=1pt] (-17.556084482653166,10.400973520971455) -- (-18.816528470369306,8.848143538703523) -- (-16.841516264109064,8.532982016427953) -- cycle;
\draw[line width=1pt] (-16.583879474603485,12.536550021123016) -- (-16.712697869356276,10.534766018775485) -- (-14.914692873057632,11.4240980176186) -- cycle;
\draw[line width=1pt] (-14.914692873057632,11.4240980176186) -- (-16.712697869356276,10.534766018775485) -- (-15.043511267810423,9.422314015271072) -- cycle;
\draw[line width=1pt] (-16.712697869356276,10.534766018775485) -- (-16.841516264109064,8.532982016427953) -- (-15.043511267810416,9.42231401527107) -- cycle;
\draw [line width=1pt] (-14.08,10.24)-- (-12.0903890116217,10.036411898863616);
\draw [line width=1pt] (-14.08,10.24)-- (-12.0903890116217,10.036411898863616);
\draw [line width=1pt] (-12.0903890116217,10.036411898863616)-- (-12.908882038318504,11.86125960901608);
\draw [line width=1pt] (-12.908882038318504,11.86125960901608)-- (-14.08,10.24);
\draw [line width=1pt] (-12.0903890116217,10.036411898863616)-- (-10.544833985948046,11.305765932462956);
\draw [line width=1pt] (-10.544833985948046,11.305765932462956)-- (-11.631628826631452,9.626814779432014);
\draw [line width=1pt] (-10.544833985948046,11.305765932462956)-- (-11.631628826631452,9.626814779432014);
\draw [line width=1pt] (-11.631628826631452,9.626814779432014)-- (-9.63421705605178,9.525098415213794);
\draw [line width=1pt] (-9.63421705605178,9.525098415213794)-- (-10.544833985948046,11.305765932462956);
\draw [line width=1pt] (-12.090389011621674,10.036411898863646)-- (-13.177183852305077,8.357460745832702);
\draw [line width=1pt] (-13.177183852305077,8.357460745832702)-- (-15.08113842418279,8.969798061554977);
\draw [line width=1pt] (-15.08113842418279,8.969798061554977)-- (-13.08113842418279,8.969798061554977);
\draw [line width=1pt] (-13.08113842418279,8.969798061554977)-- (-14.67378218931946,10.179544248724536);
\draw [line width=1pt] (-14.08,10.24)-- (-12.0903890116217,10.036411898863616);
\draw [line width=1pt] (-14.08,10.24)-- (-12.0903890116217,10.036411898863616);
\draw [line width=1pt] (-12.0903890116217,10.036411898863616)-- (-12.908882038318504,11.86125960901608);
\draw [line width=1pt] (-12.908882038318504,11.86125960901608)-- (-14.08,10.24);
\draw [line width=1pt] (-19.531096688913404,10.716135043247027)-- (-17.53197954300168,10.775554206163957);
\draw [line width=1pt] (-17.53197954300168,10.775554206163957)-- (-18.58299662051521,12.477130858206086);
\draw [line width=1pt] (-18.58299662051521,12.477130858206086)-- (-19.531096688913404,10.716135043247027);
\draw [line width=1pt] (-18.58299662051521,12.477130858206086)-- (-17.53197954300168,10.775554206163957);
\draw [line width=1pt] (-17.53197954300168,10.775554206163957)-- (-16.583879474603485,12.536550021123016);
\draw [line width=1pt] (-16.583879474603485,12.536550021123016)-- (-18.58299662051521,12.477130858206086);
\draw [line width=1pt] (-17.556084482653166,10.400973520971455)-- (-19.531096688913404,10.716135043247027);
\draw [line width=1pt] (-19.531096688913404,10.716135043247027)-- (-18.816528470369306,8.848143538703523);
\draw [line width=1pt] (-18.816528470369306,8.848143538703523)-- (-17.556084482653166,10.400973520971455);
\draw [line width=1pt] (-17.556084482653166,10.400973520971455)-- (-18.816528470369306,8.848143538703523);
\draw [line width=1pt] (-18.816528470369306,8.848143538703523)-- (-16.841516264109064,8.532982016427953);
\draw [line width=1pt] (-16.841516264109064,8.532982016427953)-- (-17.556084482653166,10.400973520971455);
\draw [line width=1pt] (-16.712697869356276,10.534766018775485)-- (-14.914692873057632,11.4240980176186);
\draw [line width=1pt] (-14.914692873057632,11.4240980176186)-- (-16.583879474603485,12.536550021123016);
\draw [line width=1pt] (-16.712697869356276,10.534766018775485)-- (-15.043511267810423,9.422314015271072);
\draw [line width=1pt] (-15.043511267810423,9.422314015271072)-- (-14.914692873057632,11.4240980176186);
\draw [line width=1pt] (-16.841516264109064,8.532982016427953)-- (-15.043511267810416,9.42231401527107);
\draw [line width=1pt] (-15.043511267810416,9.42231401527107)-- (-16.712697869356276,10.534766018775485);
\begin{scriptsize}
\draw [fill=black] (-14.08,10.24) circle (1.5pt);
\draw [fill=black] (-12.0903890116217,10.036411898863616) circle (1.5pt);
\draw [fill=black] (-12.908882038318504,11.86125960901608) circle (1.5pt);
\draw [fill=black] (-10.544833985948046,11.305765932462956) circle (1.5pt);
\draw [fill=black] (-11.631628826631452,9.626814779432014) circle (1.5pt);
\draw [fill=black] (-9.63421705605178,9.525098415213794) circle (1.5pt);
\draw [fill=black] (-12.090389011621674,10.036411898863646) circle (1.5pt);
\draw [fill=black] (-13.177183852305077,8.357460745832702) circle (1.5pt);
\draw [fill=black] (-15.08113842418279,8.969798061554977) circle (1.5pt);
\draw [fill=black] (-13.08113842418279,8.969798061554977) circle (1.5pt);
\draw [fill=black] (-14.67378218931946,10.179544248724536) circle (1.5pt);
\draw [fill=black] (-11.305113499974977,7.65364784745228) circle (1.5pt);
\draw [fill=black] (-14.08,10.24) circle (1.5pt);
\draw [fill=black] (-12.0903890116217,10.036411898863616) circle (1.5pt);
\draw [fill=black] (-12.908882038318504,11.86125960901608) circle (1.5pt);
\draw [fill=black] (-12.0903890116217,10.036411898863616) circle (1.5pt);
\draw [fill=black] (-12.908882038318504,11.86125960901608) circle (1.5pt);
\draw [fill=black] (-19.531096688913404,10.716135043247027) circle (1.5pt);
\draw [fill=black] (-17.53197954300168,10.775554206163957) circle (1.5pt);
\draw [fill=black] (-18.58299662051521,12.477130858206086) circle (1.5pt);
\draw [fill=black] (-16.583879474603485,12.536550021123016) circle (1.5pt);
\draw [fill=black] (-17.556084482653166,10.400973520971455) circle (1.5pt);
\draw [fill=black] (-18.816528470369306,8.848143538703523) circle (1.5pt);
\draw [fill=black] (-16.841516264109064,8.532982016427953) circle (1.5pt);
\draw [fill=black] (-16.712697869356276,10.534766018775485) circle (1.5pt);
\draw [fill=black] (-14.914692873057632,11.4240980176186) circle (1.5pt);
\draw [fill=black] (-15.043511267810423,9.422314015271072) circle (1.5pt);
\draw [fill=black] (-15.043511267810416,9.42231401527107) circle (2pt);
\end{scriptsize}
\end{tikzpicture}}
    \caption{A (disconnected) matchstick graph}
    \label{fig2}
\end{figure}
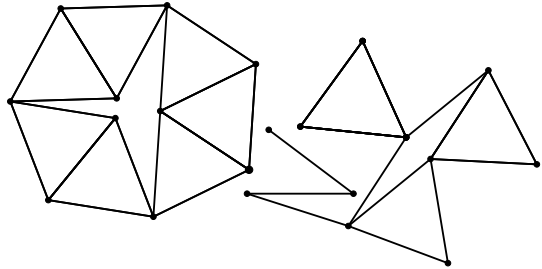
Harborth posed various problems about matchstick graphs.
The one that drew most attention in the literature is that of finding $k$-regular matchstick graphs with the smallest number of vertices.
For example, it is known that there are no $5$-regular matchstick graphs \cite{blokhuis, kurz}, and the currently smallest known $4$-regular matchstick graph, described in \cite{lighter} and known as the Harborth graph, has $52$ vertices and $104$ edges.
Another of Harborth's problems in \cite{lighter} is to find the maximum number of edges in a matchstick graph on $n$ vertices, for which he conjectured the following.
\begin{conjecture}[Harborth \cite{lighter}]\label{conj1}
For each $n\geq 1$, the maximum number of edges in a matchstick graph on $n$ vertices is $3n-\lceil\!\sqrt{12n-3}\,\rceil$.
\end{conjecture}
In \cite{harborth} Harborth proved that this is indeed the maximum number of edges if we furthermore assume that any two vertices are separated by a distance of at least $1$ (the so-called penny graphs).
This maximum is attained by an appropriately chosen set of $n$ points on the triangular lattice (Figure~\ref{fig1}).
\begin{figure}
    \centering
    \begin{tikzpicture}[line cap=round,line join=round,x=1cm,y=1cm]
\clip(7,-0.5) rectangle (12.5,4);
\draw [line width=0.8pt] (8.660254037844386,0)-- (9.660254037844386,0);
\draw [line width=0.8pt] (9.660254037844386,0)-- (10.660254037844386,0);
\draw [line width=0.8pt] (10.660254037844386,0)-- (11.160254037844386,0.8660254037844384);
\draw [line width=0.8pt] (11.160254037844386,0.8660254037844384)-- (11.660254037844386,1.7320508075688774);
\draw [line width=0.8pt] (11.660254037844386,1.7320508075688774)-- (11.160254037844386,2.598076211353316);
\draw [line width=0.8pt] (11.160254037844386,2.598076211353316)-- (10.660254037844386,3.464101615137755);
\draw [line width=0.8pt] (10.660254037844386,3.464101615137755)-- (9.660254037844386,3.4641016151377544);
\draw [line width=0.8pt] (9.660254037844386,3.4641016151377544)-- (8.660254037844387,3.4641016151377553);
\draw [line width=0.8pt] (8.660254037844387,3.4641016151377553)-- (8.160254037844386,2.5980762113533165);
\draw [line width=0.8pt] (8.160254037844386,2.5980762113533165)-- (7.6602540378443855,1.7320508075688774);
\draw [line width=0.8pt] (7.6602540378443855,1.7320508075688774)-- (8.160254037844386,0.8660254037844387);
\draw [line width=0.8pt] (8.160254037844386,0.8660254037844387)-- (8.660254037844386,0);
\draw [line width=0.8pt] (8.660254037844386,0)-- (9.160254037844386,0.8660254037844389);
\draw [line width=0.8pt] (9.160254037844386,0.8660254037844389)-- (8.160254037844386,0.8660254037844387);
\draw [line width=0.8pt] (8.160254037844386,0.8660254037844387)-- (8.660254037844386,1.7320508075688776);
\draw [line width=0.8pt] (8.660254037844386,1.7320508075688776)-- (7.6602540378443855,1.7320508075688774);
\draw [line width=0.8pt] (8.160254037844386,2.5980762113533165)-- (8.660254037844386,1.7320508075688776);
\draw [line width=0.8pt] (8.660254037844386,1.7320508075688776)-- (9.160254037844387,2.5980762113533165);
\draw [line width=0.8pt] (8.160254037844386,2.5980762113533165)-- (9.160254037844387,2.5980762113533165);
\draw [line width=0.8pt] (9.160254037844387,2.5980762113533165)-- (8.660254037844387,3.4641016151377553);
\draw [line width=0.8pt] (8.660254037844387,3.4641016151377553)-- (9.660254037844386,3.4641016151377544);
\draw [line width=0.8pt] (9.660254037844386,3.4641016151377544)-- (9.160254037844387,2.5980762113533165);
\draw [line width=0.8pt] (9.160254037844387,2.5980762113533165)-- (10.160254037844386,2.598076211353316);
\draw [line width=0.8pt] (10.160254037844386,2.598076211353316)-- (9.660254037844386,3.4641016151377544);
\draw [line width=0.8pt] (10.660254037844386,3.464101615137755)-- (10.160254037844386,2.598076211353316);
\draw [line width=0.8pt] (11.160254037844386,2.598076211353316)-- (10.160254037844386,2.598076211353316);
\draw [line width=0.8pt] (10.660254037844386,1.7320508075688772)-- (10.160254037844386,2.598076211353316);
\draw [line width=0.8pt] (10.160254037844386,2.598076211353316)-- (9.660254037844387,1.7320508075688779);
\draw [line width=0.8pt] (9.160254037844387,2.5980762113533165)-- (9.660254037844387,1.7320508075688779);
\draw [line width=0.8pt] (9.660254037844387,1.7320508075688779)-- (8.660254037844386,1.7320508075688776);
\draw [line width=0.8pt] (8.660254037844386,1.7320508075688776)-- (9.160254037844386,0.8660254037844389);
\draw [line width=0.8pt] (9.160254037844386,0.8660254037844389)-- (9.660254037844387,1.7320508075688779);
\draw [line width=0.8pt] (9.660254037844387,1.7320508075688779)-- (10.160254037844387,0.8660254037844392);
\draw [line width=0.8pt] (10.160254037844387,0.8660254037844392)-- (9.160254037844386,0.8660254037844389);
\draw [line width=0.8pt] (9.160254037844386,0.8660254037844389)-- (9.660254037844386,0);
\draw [line width=0.8pt] (9.660254037844386,0)-- (10.160254037844387,0.8660254037844392);
\draw [line width=0.8pt] (11.160254037844386,0.8660254037844384)-- (10.160254037844387,0.8660254037844392);
\draw [line width=0.8pt] (10.160254037844387,0.8660254037844392)-- (10.660254037844386,1.7320508075688772);
\draw [line width=0.8pt] (10.660254037844386,1.7320508075688772)-- (11.160254037844386,0.8660254037844384);
\draw [line width=0.8pt] (11.160254037844386,0.8660254037844384)-- (11.660254037844386,1.7320508075688774);
\draw [line width=0.8pt] (11.660254037844386,1.7320508075688774)-- (10.660254037844386,1.7320508075688772);
\draw [line width=0.8pt] (10.660254037844386,1.7320508075688772)-- (11.160254037844386,2.598076211353316);
\draw [line width=0.8pt] (9.660254037844387,1.7320508075688779)-- (10.660254037844386,1.7320508075688772);
\draw [line width=0.8pt] (10.160254037844387,0.8660254037844392)-- (10.660254037844386,0);
\begin{scriptsize}
\draw [black!50!white] (8.660254037844386,0) circle (0.5);
\draw [fill=black] (8.660254037844386,0) circle (1pt);
\draw [black!50!white] (9.660254037844386,0) circle (0.5);
\draw [fill=black] (9.660254037844386,0) circle (1pt);
\draw [black!50!white] (10.660254037844386,0) circle (0.5);
\draw [fill=black] (10.660254037844386,0) circle (1pt);

\draw [black!50!white] (8.160254037844386,0.8660254037844387) circle (0.5);
\draw [fill=black] (8.160254037844386,0.8660254037844387) circle (1pt);
\draw [black!50!white] (9.160254037844386,0.8660254037844389) circle (0.5);
\draw [fill=black] (9.160254037844386,0.8660254037844389) circle (1pt);
\draw [black!50!white] (10.160254037844387,0.8660254037844392) circle (0.5);
\draw [fill=black] (10.160254037844387,0.8660254037844392) circle (1pt);
\draw [black!50!white] (11.160254037844386,0.8660254037844384) circle (0.5);
\draw [fill=black] (11.160254037844386,0.8660254037844384) circle (1pt);

\draw [black!50!white] (7.6602540378443855,1.7320508075688774) circle (0.5);
\draw [fill=black] (7.6602540378443855,1.7320508075688774) circle (1pt);
\draw [black!50!white] (8.660254037844386,1.7320508075688776) circle (0.5);
\draw [fill=black] (8.660254037844386,1.7320508075688776) circle (1pt);
\draw [black!50!white] (9.660254037844387,1.7320508075688779) circle (0.5);
\draw [fill=black] (9.660254037844387,1.7320508075688779) circle (1pt);
\draw [black!50!white] (10.660254037844386,1.7320508075688772) circle (0.5);
\draw [fill=black] (10.660254037844386,1.7320508075688772) circle (1pt);
\draw [black!50!white] (11.660254037844386,1.7320508075688774) circle (0.5);
\draw [fill=black] (11.660254037844386,1.7320508075688774) circle (1pt);

\draw [black!50!white] (8.160254037844386,2.5980762113533165) circle (0.5);
\draw [fill=black] (8.160254037844386,2.5980762113533165) circle (1pt);
\draw [black!50!white] (9.160254037844387,2.5980762113533165) circle (0.5);
\draw [fill=black] (9.160254037844387,2.5980762113533165) circle (1pt);
\draw [black!50!white] (10.160254037844386,2.598076211353316) circle (0.5);
\draw [fill=black] (10.160254037844386,2.598076211353316) circle (1pt);
\draw [black!50!white] (11.160254037844386,2.598076211353316) circle (0.5);
\draw [fill=black] (11.160254037844386,2.598076211353316) circle (1pt);

\draw [black!50!white] (8.660254037844387,3.4641016151377553) circle (0.5);
\draw [fill=black] (8.660254037844387,3.4641016151377553) circle (1pt);
\draw [black!50!white] (9.660254037844386,3.4641016151377544) circle (0.5);
\draw [fill=black] (9.660254037844386,3.4641016151377544) circle (1pt);
\draw [black!50!white] (10.660254037844386,3.464101615137755) circle (0.5);
\draw [fill=black] (10.660254037844386,3.464101615137755) circle (1pt);
\end{scriptsize}
\end{tikzpicture}
    \caption{A penny graph attaining equality in Conjecture~\ref{conj1}}
    \label{fig1}
\end{figure}
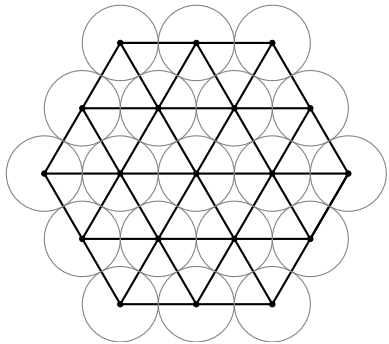
In \cite[p.~225]{research} this problem is again mentioned, where it is stated that it ``seems very likely that the maximum number of edges in a crossing-free unit-distance graph is again $\lfloor 3n-\sqrt{12n-3}\rfloor$\dots''
It is not hard to show from the Euler formula and the isoperimetric inequality that the number of edges is at most $3n-\sqrt{\rule[-0.1em]{0pt}{1em}\smash{2\pi\sqrt{\smash[b]{3}}\cdot n}}+O(1)$.
Our main result is the following improvement.
\begin{theorem}\label{thm3}
In a matchstick graph with $n$ vertices, the number $e$ of edges satisfies
\[\textstyle e\leq 3n - c\sqrt{n-1/4},\]
where $c=\frac12\Bigl(\!\sqrt{12}+\sqrt{\rule[-0.05em]{0pt}{0.9em}\smash{2\pi\sqrt{\smash[b]{3}}}}\Bigr)\approx 3.3815\dots$.
%$c=\frac12\left(\!\sqrt{12}+\sqrt{2\pi\sqrt{3}}\,\right)\approx 3.3815\dots$.
\end{theorem}

As a calculation shows, this bound turns out to be strong enough that it settles Conjecture~\ref{conj1} for all $n\leq 14$, as well as for
\begin{align*}
n&=16, 18, 19, 20, 21, 23, 24, 26, 27, 29, 30, 32, 33, 36, 37, 39, 40,\\
& 43, 44, 47, 48, 51, 52, 55, 56, 60, 61, 65, 69, 70, 74, 75, 79, 80, 85,\\
& 90, 91, 96, 102, 108, 114, 120, 127.
\end{align*}
% Put table[(n,floor(3n-(√(2π√3)+√12)/2*√(n-1/4))-floor(3n-√(12n-3))),{n,1,200}] into wolframalpha.com
% 
The proof of Theorem~\ref{thm3} uses the isoperimetric inequality (Lemma~\ref{isoperimetric} below), as well as the following result which bounds the number of edges in terms of the number of bounded non-triangular faces of the matchstick graph.
Its proof is based on Harborth's induction proof \cite{harborth} of Conjecture~\ref{conj1} for penny graphs.
\begin{theorem}\label{thm1}
In a matchstick graph with $n$ vertices, $e$ edges, and $g$ bounded non-triangular faces, we have
\[e \leq 3n-\sqrt{12n-3}+g.\]
\end{theorem}
We do not assume that the graph is $2$-connected or even connected in Theorem~\ref{thm1}.
This is not just for the sake of generality, as this general statement is needed when the induction hypothesis is applied in the proof.

Even though Theorem~\ref{thm1} does not seem to be sharp if there are bounded non-triangular faces ($g > 0$), when combined with the Euler formula, it gives the following sharp upper bound for the number of triangular faces of a matchstick graph.
\begin{corollary}\label{cor1}
In a matchstick graph with $n$ vertices, the number %$f_3$
of bounded triangular faces is at most $2n+1-\sqrt{12n-3}$.
%\[f_3 \leq 2n+1-\sqrt{12n-3}.\]
\end{corollary}
Note that for each $n\geq 1$ there is a matchstick graph on the triangular lattice with $n$ vertices and $2n+1-\lceil\!\sqrt{12n-3}\,\rceil$ triangular faces (Figure~\ref{fig1}).

The above results are proved in Section~\ref{section:proofs}.
In the next section, we establish our terminology and introduce the fundamental tools we'll need: the Euler formula, a double-counting identity, and the isoperimetric inequality.

\section{Plane graphs and matchstick graphs}
A  \emph{plane graph} $G=(V,E)$ is defined to be a drawing of a graph in the plane such that each vertex $v\in V$ is a different point in the plane, and each edge $uv\in E$ is represented by a simple arc joining $u$ and $v$, in such a way that two arcs only intersect in a common endpoint.
The \emph{faces} of a plane graph are the connected components of the complement of the plane graph in the plane.
One of the faces is unbounded.
Throughout this paper, we will denote the number of vertices by $n$, the number of edges by $e$, and the number of bounded faces by $f$.

By the Euler formula, whenever $G$ is connected, we have $n-e+f=1$.
If $G$ is furthermore $2$-connected, then each face is bounded by a cycle.
Denote the number of vertices of the unbounded face by $b$, and the number of bounded faces with exactly $i$ boundary vertices by $f_i$.
We have the following well-known relation.
\begin{lemma}\label{euler_double}
For any $2$-connected plane graph with $n$ vertices, $e$ edges, $b$ boundary vertices, and $f_i$ bounded faces with $i$ vertices, $i\geq 3$, we have $e = 3n-3-b-\sum_{i\geq 4}(i-3)f_i$.
\end{lemma}
\begin{proof}
If we add up the number of vertices of each face, including the unbounded face, we count each edge twice, thus obtaining $2e=b+\sum_{i\geq 3}if_i$.
By Euler's formula, $3e=3n+3f-3$.
Subtracting these two identities, we obtain the result.
\end{proof}

A \emph{matchstick graph} is a plane graph in which each edge is represented by a straight-line segment of unit length.
A matchstick graph is called a \emph{penny graph} if the distance between any two vertices is at least $1$, and there is an edge between all pairs of vertices at distance $1$.

In the proof of Theorem~\ref{thm3} we will need the following consequence of the isoperimetric inequality that asserts that among all simple closed curves in the plane of a fixed length, the circle is the unique curve that encloses the largest area \cite{blasjo}.
\begin{lemma}\label{isoperimetric}
Let $G$ be a $2$-connected matchstick graph with $b$ vertices on the outer boundary and $f_3$ bounded triangular faces.
Then $b^2 > \pi\sqrt{3}f_3$.
\end{lemma}
\begin{proof}
The polygon bounding the unbounded face has $b$ edges and encloses $f_3$ equilateral triangles of unit side length.
Each of these triangles has area $\sqrt{3}/4$.
Thus the polygon has area $A\geq \frac{\sqrt{3}}{4}f_3$ and perimeter $b$.
By the isoperimetric inequality, any region of area $A$ bounded by a simple closed curve of length $b$ satisfies $b^2\geq 4\pi A$, with equality only if the curve is a circle.
It follows that $b^2 > 4\pi A\geq \pi\sqrt{3}f_3$.
\end{proof}

\section{Proofs}\label{section:proofs}
We will repeatedly use the following inequality involving sums of square roots.
\begin{lemma}\label{sqrt}
Let $\alpha,\beta,\gamma,\delta$ be non-negative real numbers with $\beta\leq\alpha\leq\gamma$ and $\alpha+\delta=\beta+\gamma$.
Then \[ \sqrt{\beta}+\sqrt{\gamma}\leq \sqrt{\alpha}+\sqrt{\delta},\]
with equality iff $\alpha\in\{\beta,\gamma\}$.
\end{lemma}
\begin{proof}
From
\[ 0\leq (\alpha-\beta)(\gamma-\alpha) = \alpha\delta-\beta\gamma\]
we obtain \[ (\sqrt{\beta}+\sqrt{\gamma})^2 = \beta+\gamma+2\sqrt{\beta\gamma}\leq \alpha+\delta+2\sqrt{\alpha\delta}=(\sqrt{\alpha}+\sqrt{\delta})^2.\qedhere\]
\end{proof}

\begin{lemma}\label{sqrt2}
If $n,n_1,n_2$ are integers such that $n_1,n_2\geq 3$, $n_1+n_2=n+2$, then
\[\lfloor 3n-\sqrt{12n-3}\rfloor + 1\geq \lfloor 3n_1-\sqrt{12n_1-3}\rfloor+\lfloor 3n_2-\sqrt{12n_2-3}\rfloor.\]
\end{lemma}
\begin{proof}
The required inequality is equivalent to
\begin{equation}\label{s0}
\sqrt{12n-3}+5\leq\lceil\!\sqrt{12n_1-3}\,\rceil+\lceil\!\sqrt{12n_2-3}\,\rceil.
\end{equation}
Without loss of generality, $n_2\geq n_1$.
If $n_1\geq 6$, then we set $\alpha=12n_1-3$, $\beta=69$, $\gamma=12n-51$, $\delta=12n_2-3$ in Lemma~\ref{sqrt} and obtain
\begin{equation}\label{s1}
\sqrt{69}+\sqrt{12n-51}\leq\sqrt{12n_1-3}+\sqrt{12n_2-3}.
\end{equation}
Apply Lemma~\ref{sqrt} again with $\alpha=12n-51$, $\beta=33$, $\gamma=12n-3$, $\delta=81$ (since $n\geq 7$), we get
\begin{equation}\label{s2}
\sqrt{12n-3}+\sqrt{33}\leq\sqrt{12n-51}+9.
\end{equation}
It follows from \eqref{s1} and \eqref{s2} that \[\sqrt{12n-3}+\sqrt{69}+\sqrt{33}-9\leq\sqrt{12n_1-3}+\sqrt{12n_2-3}.\]
Finally, a calculation shows that $\sqrt{69}+\sqrt{33}-9>5$, and \eqref{s0} follows.

For the remaining cases $n_1=3,4,5$, we need to round $\sqrt{12n_1-3}$ up to obtain \eqref{s0}.
For instance, if $n_1=3$, then we need to show $\sqrt{12n-3}+5\leq 6+\sqrt{12n-15}$, which follows from Lemma~\ref{sqrt} by setting $\alpha=12n-15$, $\beta=33$, $\gamma=12n-3$, $\delta=45$ (since $n\geq 4$), and noting that $\sqrt{45}-\sqrt{33} < 1$.
Similarly, when $n_1=4$ then $n\geq 5$, and we obtain $\sqrt{12n-3}+5\leq 7+\sqrt{12n-27}$ by setting $\alpha=12n-27$, $\beta=25$, $\gamma=12n-3$, $\delta=49$,
and when $n_1=5$ then $n\geq 6$, and $\sqrt{12n-3}+5\leq 8+\sqrt{12n-39}$ follows by setting $\alpha=12n-39$, $\beta=25$, $\gamma=12n-3$, $\delta=61$.
\end{proof}

\begin{proof}[Proof of Theorem~\ref{thm1}]
We use induction on the number of vertices $n\geq 1$.
The theorem clearly holds when $n=1$ or $n=2$.
We now assume that $n\geq 3$ and that the theorem holds for all smaller values of $n$ as induction hypothesis.

If the matchstick graph $G$ is not connected, let $G'$ be a connected component of $G$.
If $G'$ is in a bounded face of $G-G'$, we can move $G'$ to the unbounded face of $G-G'$.
Note that this does not change the number of non-triangular faces, unless $G'$ was originally inside a triangular face of $G-G'$.
However, then $G'$ cannot have any edges, and we are done by applying induction to $G-G'$.
Thus we may assume that neither $G'$ nor $G-G'$ lies in a bounded face of the other graph.
Then it is easy to move $G'$ so that one of its vertices is at distance $1$ from a vertex of $G-G'$, while keeping $G'$ and $G-G'$ disjoint.
This creates a new edge that joins two connected components of $G$.
This process can be repeated until $G$ is connected, without decreasing the number of edges.

We now assume without loss of generality that $G$ is connected.
If $G$ is not $2$-connected, then there is a vertex $v$ such that $G-v$ is disconnected (Figure~\ref{fig:cutvertex}).
We can then decompose $G$ into two induced subgraphs $G_1$ and $G_2$ having only $v$ in common.
If $G_i$ has $n_i$ vertices, $e_i$ edges, and $g_i$ non-triangular faces ($i=1,2$), then $n_1,n_2\geq 2$, $n_1+n_2=n+1$, and $e_1+e_2=e$.
It is clear that $g_1+g_2=g$ if $G_1$ lies in the unbounded face of $G_2$ and $G_2$ in the unbounded face of $G_1$.
Suppose that $G_1$ (say) lies in a bounded face of $G_2$.
Then this face cannot be a triangle, as then $G_1$ would not have any edges, contradicting the connectedness of $G$.
It follows that $g_1+g_2=g$ in this case too.
By induction,
\begin{align*}
e=e_1+e_2 &\leq 3n_1-\sqrt{12n_1-3}+g_1+3n_2-\sqrt{12n_2-3}+g_2\\
&=3n+3-\sqrt{12n_1-3}-\sqrt{12n_2-3}+g\\
&\leq 3n-\sqrt{12n-3}+g,
\end{align*}
by Lemma~\ref{sqrt} with $\alpha=12n_1-3$, $\beta=9$, $\gamma=12n-3$, $\delta=12n_2-3$.

For the remainder of the proof we assume without loss of generality that $G$ is $2$-connected.
In particular, the boundary of the unbounded face is a cycle.
\begin{figure}
        \centering
        \subfigure[Cut vertex $v$]{\scalebox{1.5}{\begin{tikzpicture}[line cap=round,line join=round,x=1cm,y=1cm]
\clip(-6.3,-0.8) rectangle (-3.8,0);
\draw [rotate around={-175.1999069520851:(-4.603117370568568,-0.3857946016542552)},line width=0.8pt, fill=black!10!white] (-4.603117370568568,-0.3857946016542552) ellipse (0.36221476214560616cm and 0.34609338541412427cm);
\draw [rotate around={169.5898401359414:(-5.313332797030034,-0.3716194315463852)},line width=0.8pt, fill=black!10!white] (-5.313332797030034,-0.3716194315463852) ellipse (0.34898325666851265cm and 0.33296679724394357cm);
\begin{scriptsize}
\draw [fill=black] (-4.964882810742803,-0.3734463178698877) circle (1pt);
\draw[color=black] (-4.8,-0.4) node {$v$};
%\draw [fill=black] (-2.6274440485619386,-0.03525448556333732) circle (1pt);
%\draw[color=black] (-2.6156156408429885,0.09963793540070359) node {$u$};
%\draw [fill=black] (-2.6068883395775253,-0.482814750656144) circle (1pt);
%\draw[color=black] (-2.5971538163251116,-0.6665277820911808) node {$v$};
%\draw [fill=black] (-0.6264401116114577,0.14010845617564968) circle (1pt);
%\draw[color=black] (-0.6632776980775241,0.2334861631553099) node {$u$};
%\draw [fill=black] (-0.6340686066217696,-0.29471575941211225) circle (1pt);
%\draw [fill=black] (-0.42047074633304354,-0.6379980348761357) circle (1pt);
%\draw[color=black] (-0.3817348741799039,-0.8142223782341946) node {$v$};
%\draw [fill=black] (-0.229758371075253,0.025681031020975492) circle (1pt);
%\draw [fill=black] (-0.10770245091026698,-0.34811522448429355) circle (1pt);
\end{scriptsize}
\end{tikzpicture}}\label{fig:cutvertex}}
        \subfigure[Chord $uv$]{\scalebox{1.5}{\begin{tikzpicture}[line cap=round,line join=round,x=1cm,y=1cm]
\clip(-3.6,-0.8) rectangle (-1.7,0.25);
%\draw [rotate around={-175.1999069520851:(-4.603117370568568,-0.3857946016542552)},line width=0.8pt] (-4.603117370568568,-0.3857946016542552) ellipse (0.36221476214560616cm and 0.34609338541412427cm);
%\draw [rotate around={169.5898401359414:(-5.313332797030034,-0.3716194315463852)},line width=0.8pt] (-5.313332797030034,-0.3716194315463852) ellipse (0.34898325666851265cm and 0.33296679724394357cm);
\draw [shift={(-2.272925671792888,-0.2432242154702259)},line width=0.8pt, fill=black!10!white]  plot[domain=-2.519272738886633:2.611064952584812,variable=\t]({1*0.4110166517606151*cos(\t r)+0*0.4110166517606151*sin(\t r)},{0*0.4110166517606151*cos(\t r)+1*0.4110166517606151*sin(\t r)});
\draw [shift={(-2.9078719897519694,-0.2723862585528154)},line width=0.8pt, fill=black!10!white]  plot[domain=0.7019380475791273:5.673039473298639,variable=\t]({1*0.3672482919786187*cos(\t r)+0*0.3672482919786187*sin(\t r)},{0*0.3672482919786187*cos(\t r)+1*0.3672482919786187*sin(\t r)});
\draw [line width=0.8pt] (-2.6274440485619386,-0.03525448556333732)-- (-2.6068883395775253,-0.482814750656144);
%\draw [line width=0.8pt] (-0.6264401116114577,0.14010845617564968)-- (-0.6340686066217696,-0.29471575941211225);
%\draw [line width=0.8pt] (-0.6340686066217696,-0.29471575941211225)-- (-0.42047074633304354,-0.6379980348761357);
%\draw [line width=0.8pt] (-0.6264401116114577,0.14010845617564968)-- (-0.229758371075253,0.025681031020975492);
%\draw [line width=0.8pt] (-0.229758371075253,0.025681031020975492)-- (-0.10770245091026698,-0.34811522448429355);
%\draw [line width=0.8pt] (-0.10770245091026698,-0.34811522448429355)-- (-0.42047074633304354,-0.6379980348761357);
%\draw [shift={(-0.7205808947686879,-0.3011250597081236)},line width=0.8pt]  plot[domain=1.3605899681101508:5.440137122996831,variable=\t]({1*0.45116460698132405*cos(\t r)+0*0.45116460698132405*sin(\t r)},{0*0.45116460698132405*cos(\t r)+1*0.45116460698132405*sin(\t r)});
%\draw [shift={(-0.21028666198043033,-0.1660471745582904)},line width=0.8pt]  plot[domain=-1.9897779469522128:2.507319730879609,variable=\t]({1*0.5166381362906737*cos(\t r)+0*0.5166381362906737*sin(\t r)},{0*0.5166381362906737*cos(\t r)+1*0.5166381362906737*sin(\t r)});
\begin{scriptsize}
%\draw [fill=black] (-4.964882810742803,-0.3734463178698877) circle (1pt);
%\draw[color=black] (-4.844880951376604,-0.43113951948825247) node {$v$};
\draw [fill=black] (-2.6274440485619386,-0.03525448556333732) circle (1pt);
\draw[color=black] (-2.67,0.18) node {$u$};
\draw [fill=black] (-2.6068883395775253,-0.482814750656144) circle (1pt);
\draw[color=black] (-2.5971538163251116,-0.7065277820911808) node {$v$};
%\draw [fill=black] (-0.6264401116114577,0.14010845617564968) circle (1pt);
%\draw[color=black] (-0.6632776980775241,0.4334861631553099) node {$u$};
%\draw [fill=black] (-0.6340686066217696,-0.29471575941211225) circle (1pt);
%\draw [fill=black] (-0.42047074633304354,-0.6379980348761357) circle (1pt);
%\draw[color=black] (-0.3817348741799039,-0.8142223782341946) node {$v$};
%\draw [fill=black] (-0.229758371075253,0.025681031020975492) circle (1pt);
%\draw [fill=black] (-0.10770245091026698,-0.34811522448429355) circle (1pt);
\end{scriptsize}
\end{tikzpicture}}\label{fig:chord}}
        \subfigure[Face with non-neighbouring vertices $u$ and $v$ on the boundary]{\scalebox{1.5}{\begin{tikzpicture}[line cap=round,line join=round,x=1cm,y=1cm]
\clip(-1.9,-0.9) rectangle (1,0.5);
\draw [shift={(-0.7205808947686879,-0.3011250597081236)},line width=0.8pt, fill=black!10!white]  plot[domain=1.3605899681101508:5.440137122996831,variable=\t]({1*0.45116460698132405*cos(\t r)+0*0.45116460698132405*sin(\t r)},{0*0.45116460698132405*cos(\t r)+1*0.45116460698132405*sin(\t r)});
\draw [shift={(-0.21028666198043033,-0.1660471745582904)},line width=0.8pt, fill=black!10!white]  plot[domain=-1.9897779469522128:2.507319730879609,variable=\t]({1*0.5166381362906737*cos(\t r)+0*0.5166381362906737*sin(\t r)},{0*0.5166381362906737*cos(\t r)+1*0.5166381362906737*sin(\t r)});
\draw [line width=0.8pt, fill=white] (-0.6264401116114577,0.14010845617564968)-- (-0.6340686066217696,-0.29471575941211225)--
(-0.42047074633304354,-0.6379980348761357)--
(-0.10770245091026698,-0.34811522448429355)--
(-0.229758371075253,0.025681031020975492)--cycle;
\begin{scriptsize}
%\draw [fill=black] (-4.964882810742803,-0.3734463178698877) circle (1pt);
%\draw[color=black] (-4.844880951376604,-0.43113951948825247) node {$v$};
%\draw [fill=black] (-2.6274440485619386,-0.03525448556333732) circle (1pt);
%\draw[color=black] (-2.7156156408429885,0.2063793540070359) node {$u$};
%\draw [fill=black] (-2.6068883395775253,-0.482814750656144) circle (1pt);
%\draw[color=black] (-2.5971538163251116,-0.7065277820911808) node {$v$};
\draw [fill=black] (-0.6264401116114577,0.14010845617564968) circle (1pt);
\draw[color=black] (-0.7,0.31) node {$u$};
\draw [fill=black] (-0.6340686066217696,-0.29471575941211225) circle (1pt);
\draw [fill=black] (-0.42047074633304354,-0.6379980348761357) circle (1pt);
\draw[color=black] (-0.3817348741799039,-0.8142223782341946) node {$v$};
\draw [fill=black] (-0.229758371075253,0.025681031020975492) circle (1pt);
\draw [fill=black] (-0.10770245091026698,-0.34811522448429355) circle (1pt);
\end{scriptsize}
\end{tikzpicture}}\label{fig:face}}
        \caption{Three cases in the proof of Theorem~\ref{thm1}}
        \label{fig:cut-chord-face}
\end{figure}
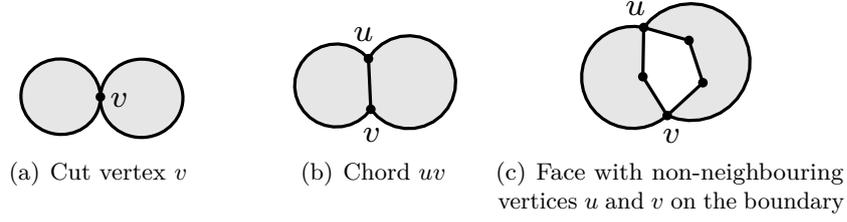
Suppose that the boundary cycle has a chord (Figure~\ref{fig:chord}).
Let $G_1$ and $G_2$ be two induced subgraphs covering $G$ such that the chord is their only common edge and the endpoints of the chord their only two common vertices.
Using the same notation as before, with $G_i$ having $n_i$ vertices, $e_i$ edges, and $g_i$ non-triangular faces ($i=1,2$), we now have $n_1,n_2\geq 3$, $n_1+n_2=n+2$, $e_1+e_2=e+1$, and $g_1+g_2=g$.
Again we use induction on these two subgraphs to obtain
\begin{align*}
e=e_1+e_2 -1 &\leq \lfloor 3n_1-\sqrt{12n_1-3}\rfloor + g_1 + \lfloor 3n_2-\sqrt{12n_2-3}\rfloor + g_2 -1\\
&\leq 3n-\sqrt{12n-3}+g\quad\text{by Lemma~\ref{sqrt2}.}
\end{align*}

From now on we assume without loss of generality that the boundary cycle does not have a chord.

We next show that
if a bounded non-triangular face shares more than one vertex with the unbounded face, then we can assume without loss of generality that only two of its vertices are on the boundary cycle, they are adjacent, and at most one of the two interior angles of the face at these two vertices is smaller than $60^\circ$.

First suppose that a non-triangular bounded face has two non-adjacent vertices on the boundary (Figure~\ref{fig:face}).
We decompose $G$ into two induced subgraphs $G_1$ and $G_2$ such they only have these two vertices in common, with no common edge.
Again using the notation with $G_i$ having $n_i$ vertices, $e_i$ edges, and $g_i$ non-triangular faces ($i=1,2$), we have $n_1,n_2\geq 3$, $n_1+n_2=n+2$, $e_1+e_2=e$, and $g_1+g_2=g-1$.
By induction,
\begin{align*}
e=e_1+e_2 &\leq \lfloor 3n_1-\sqrt{12n_1-3}\rfloor + g_1 + \lfloor 3n_2-\sqrt{12n_2-3}\rfloor + g_2\\
&\leq 3n-\sqrt{12n-3}+g \quad\text{by Lemma~\ref{sqrt2}.}
\end{align*}
From now on we assume without loss of generality that whenever a non-triangular bounded face has more than one vertex on the boundary, it has only two vertices on the boundary and they are adjacent.
If both angles at these boundary vertices are $< 60^\circ$, then two of the edges of the face will be forced to intersect.
It follows that at most one angle of a non-triangular face is less than $60^\circ$.

Let $g_b$ denote the number of non-triangular bounded faces that share a vertex with the unbounded face.
At each vertex of of the boundary cycle of degree $i$ there are $i-1$ angles interior to bounded faces.
If we denote the number of boundary vertices of degree $i$ by $b_i$, then
the total number of boundary vertices is $b=\sum_{i\geq 2}b_i$, the number of angles is $\sum_{i\geq 2}(i-1)b_i$, and the sum of these angles equals $180^\circ(b-2)$.
By the previous paragraph we have that at most $g_b$ of these angles are smaller than $60^\circ$.
Therefore, $180^\circ(b-2)\geq 60^\circ(\sum_{i\geq 2}(i-1)b_i-g_b)$, hence
\begin{equation}\label{two}
\sum_{i\geq 2}(i-1)b_i -g_b\leq 3b-6.
\end{equation}
Since the boundary cycle has no chord, when we remove the $b$ vertices on the boundary together with their incident edges, we remove exactly $\sum_{i\geq 2}(i-1)b_i$ edges.
We also remove $g_b$ non-triangular faces and $b$ vertices.
Without loss of generality, $n-b>0$, otherwise $G$ is just the boundary cycle, hence $e=n\leq 3n-\sqrt{12n-3}+g$ since $n\geq 3$.
We use induction on the remaining graph of $n-b$ vertices to obtain
\[e-\sum_{i\geq 2}(i-1)b_i \leq 3(n-b)-\sqrt{12(n-b)-3}+g-g_b,\] hence
\begin{align*}
e&\leq 3(n-b)-\sqrt{12(n-b)-3}+g-g_b+\sum_{i\geq 2}(i-1)b_i\\
&\leq 3(n-b)-\sqrt{12(n-b)-3}+g+3b-6\quad\text{by \eqref{two}}\\
&= 3n-\sqrt{12(n-b)-3}+g-6,
\end{align*}
and in order to conclude that $e\leq 3n-\sqrt{12n-3}+g$, we need to show that
\begin{equation}\label{three}
\sqrt{12n-3}\leq\sqrt{12(n-b)-3}+6.
\end{equation}
If $b\leq\sqrt{12n-3}-3$, then
\[ \sqrt{12(n-b)-3} \geq \sqrt{12\bigl(n-\sqrt{12n-3}+3\bigr)-3}
=\sqrt{\bigl(\sqrt{12n-3}-6\bigr)^2},\]
which gives \eqref{three}, since $n > b\geq 3$, hence $n\geq 4$ and $\sqrt{12n-3} - 6 > 0$.
Otherwise, $b > \sqrt{12n-3}-3$, and by Lemma~\ref{euler_double}, $e\leq 3n-b-3 < 3n-\sqrt{12n-3}$.
\end{proof}

\begin{remark}
In the last step of the above proof, there is some slack, as we actually have $e\leq 3n-b-3-g$ from Lemma~\ref{euler_double}.
By taking this into account, it is possible to prove the slightly stronger inequality $e\leq3n-\sqrt{12(n+2g)-3}+g$.
However, there are then more boundary cases to deal with, and as this is not much of an improvement, we settled for the weaker inequality in Theorem~\ref{thm1}.
\end{remark}

\begin{remark}
When we removed the outer boundary cycle in the above proof, we needed this cycle not to have a chord in order to count the number of edges that are removed.
This point is overlooked in Harborth's original proof \cite{harborth} on which this proof is based.
\end{remark}

\begin{proof}[Proof of Corollary~\ref{cor1}]
By the Euler formula, $n-e+f_3+g=1$, and by Theorem~\ref{thm1}, $e-g\leq 3n-\sqrt{12n-3}$.
It follows that $f_3\leq 2n+1-\sqrt{12n-3}$.
\end{proof}

\begin{proof}[Proof of Theorem~\ref{thm3}]
As in the proof of Theorem~\ref{thm1} we use induction on $n$.
The theorem is easy to verify for $n=1, 2$, and as in the proof of Theorem~\ref{thm1}, we can assume that $G$ is connected.

To show that we can furthermore assume that $G$ is $2$-connected, we also proceed as in the proof of Theorem~\ref{thm1}.
If $G$ is not $2$-connected, we can decompose $G$ into two induced subgraphs $G_1$ and $G_2$ having only a single vertex in common.
Let $G_i$ have $n_i$ vertices and $e_i$ edges $(i=1,2)$.
Then $n_1,n_2\geq 2$, $n_1 + n_2 = n+1$ and $e_1 + e_2 = e$.
By induction and using the shorthand  $c=\frac12\Bigl(\!\sqrt{12}+\sqrt{\rule[-0.05em]{0pt}{0.9em}\smash{2\pi\sqrt{\smash[b]{3}}}}\Bigr)$
%$c=\frac12\left(\!\sqrt{12} + \sqrt{\vphantom{\bigl\{\bigr\}}\smash[t]{2\pi\sqrt{3}}}\,\right)$
we have
\begin{align*}
e=e_1+e_2 &\leq 3n_1-c\sqrt{n_1-1/4}+3n_2-c\sqrt{n_2-1/4}\\
&=3n+3-c\sqrt{n_1-1/4}-c\sqrt{n_2-1/4}.
\end{align*}
To conclude that $e \leq 3n - c\sqrt{n-1/4}$ we need to show that
\begin{equation*}
\frac{3}{c} + \sqrt{n-1/4} \leq \sqrt{n_1 - 1/4} + \sqrt{n_2 -1/4}.
\end{equation*}
Since $n_i \geq 2$, we can apply Lemma~\ref{sqrt} with $\alpha=n_1-1/4$, $\beta=3/2$, $\gamma=n-1$, $\delta=n_2-1/4$ to get
\begin{equation*}
    \sqrt{3/2} + \sqrt{n-1} \leq \sqrt{n_1-1/4}+\sqrt{n_2-1/4}.
\end{equation*}
Since $n\geq 3$, we can again apply Lemma~\ref{sqrt} with $\alpha=n-1$, $\beta=5/4$, $\gamma=n-1/4$, $\delta=2$, to obtain
\begin{equation*}
   \sqrt{5}/2 + \sqrt{n-1/4} \leq \sqrt{n-1}+\sqrt{2}.
\end{equation*}
Combining these two equations together gives
\begin{equation*}
    \sqrt{n-1/4} + \sqrt{5}/2 -\sqrt{2} + \sqrt{3/2} \leq \sqrt{n_1-1/4}+\sqrt{n_2-1/4}
\end{equation*}
which shows the required inequality since $\sqrt{5}/2 -\sqrt{2} + \sqrt{3/2} > 3/c$.

We now assume that $G$ is 2-connected.
Thus the unbounded face is bounded by a cycle with $b$ edges.
As before, denote the number of bounded faces with $i$ vertices by $f_i$ ($i\geq 3$), and the number of non-triangular faces by $g$.
By Lemma~\ref{euler_double}, noting that $g =\sum_{i\geq 4} f_i \leq \sum_{i \geq 4}(i-3)f_i$, we have
\begin{equation}\label{eqn:e<g}
e \leq 3n -3 -b -g.
\end{equation}
Suppose for the sake of contradiction that
\begin{equation}\label{eqn:assum}
e > 3n - c\sqrt{n-1/4}.
\end{equation}
Then \eqref{eqn:e<g} and \eqref{eqn:assum} give the following upper bound for $g$:
\begin{align*}
    g &< c\sqrt{n-1/4} -3 -b. \stepcounter{equation} \tag{\theequation} \label{eqn:g<}
\end{align*}
We obtain the following lower bound for $f_3$ from \eqref{eqn:assum}, \eqref{eqn:g<}, and the Euler formula:
\begin{align*}
    f_3&=e-n-g+1 \\
    &> 3n - c\sqrt{n-1/4} -n -c\sqrt{n-1/4} +3 +b +1 \\
    &= 2n - 2c\sqrt{n-1/4} +4 +b.
\end{align*}
Substitute this into the inequality $b^2 > \pi\sqrt{3}f_3$ from Lemma~\ref{isoperimetric} to obtain
\begin{align*}
    b^2 -\pi\sqrt{3}b &> 2\pi\sqrt{3}\left(n-c\sqrt{n-1/4}+2\right).
\end{align*}
By completing the square,
\begin{align*}
    \left(b-\frac{\pi\sqrt{3}}{2} \right)^2 &> \frac{3\pi^2}{4}+2\pi\sqrt{3}\left(n-c\sqrt{n-1/4}+2\right),
\end{align*}
we get the following lower bound for $b$:
\begin{align*}
    b &> \frac{\pi \sqrt{3}}{2} + \sqrt{\frac{3\pi^2}{4}+ 2\pi\sqrt{3}\left(n - c\sqrt{n-1/4}+2\right)}.
\end{align*}
We would like to deduce from this that $b\geq \sqrt{2\pi\sqrt{3}\left(n-1/4\right)}-3$.
It is sufficient to show the following:
\begin{equation}\label{eq2}
    \sqrt{\frac{3\pi^2}{4}+2\pi\sqrt{3} \left(n-c\sqrt{n-1/4}+2\right)} \geq \sqrt{2\pi\sqrt{3}\left(n-1/4\right)} - \frac{\pi\sqrt{3}}{2}-3.
\end{equation}
Since the left-hand side is non-negative, we can assume without loss of generality that
\begin{equation}\label{eq3}
\sqrt{2\pi\sqrt{3}\left(n-1/4\right)} \geq \frac{\pi\sqrt{3}}{2} + 3.
\end{equation}
Then we can square both sides of \eqref{eq2} and rearrange to obtain the equivalent
\begin{equation*}
    \left(\pi\sqrt{3}+6-c\sqrt{\vphantom{\bigl\{\bigr\}}\smash[t]{2\pi\sqrt{3}}}\right)\sqrt{2\pi\sqrt{3}\left(n-1/4\right)} \geq 9-\frac{3\pi\sqrt{3}}{2}.
\end{equation*}
This follows from \eqref{eq3}, upon checking that $\pi\sqrt{3}+6-c\sqrt{\vphantom{\bigl\{\bigr\}}\smash[t]{2\pi\sqrt{3}}}>0$ and
\[ \left(\!\pi\sqrt{3}+6-c\sqrt{\vphantom{\bigl\{\bigr\}}\smash[t]{2\pi\sqrt{3}}}\,\right)\left(\frac{\pi\sqrt{3}}{2} + 3\right) \geq 9-\frac{3\pi\sqrt{3}}{2}.\]
So we have shown that $b \geq \sqrt{\vphantom{\bigl\{\bigr\}}\smash[t]{2\pi\sqrt{3}\left(n-1/4\right)}}-3$, which, together with \eqref{eqn:e<g} gives
\begin{align*}
e &\leq 3n - \sqrt{\vphantom{\bigl\{\bigr\}}\smash[t]{2\pi\sqrt{3}\left(n-1/4\right)}}-g.
\end{align*}
By Theorem~\ref{thm1} we also have $e\leq 3n-\sqrt{12n-3}+g$.
Adding these two bounds, we obtain $e\leq 3n - c \sqrt{n-1/4}$,
which contradicts the assumption \eqref{eqn:assum}.
Thus the assumption \eqref{eqn:assum} is false and the theorem follows.
\end{proof}

\begin{remark}
Eppstein \cite{eppstein} uses the isoperimetric inequality to find an upper bound of the form $2n-c\sqrt{n}$ for the number of edges in a triangle-free penny graph on $n$ vertices.
To show an upper bound of this form for triangle-free matchstick graphs will need a new idea, as there is no obvious way to bound the area of the bounded faces from below.
\end{remark}


\begin{thebibliography}{7}
    
    \bibitem{blasjo}
    Bl{\aa}sj\"o, V. (2005).
    \newblock \emph{The Isoperimetric Problem}.
    \newblock Amer.\ Math.\ Monthly \textbf{112}, 526--566.
    
    \bibitem{blokhuis}
    Blokhuis, A. (1982).
    \newblock \emph{Regular finite planar maps with equal edges}.
    \newblock arXiv preprint arXiv:1401.1799
    
    \bibitem{research}
    Brass, P., W.~O.~J. Moser, and J.~Pach (2005).
    \newblock \emph{Research Problems in Discrete Geometry}.
    \newblock Springer-Verlag, New York.
    
    \bibitem{eppstein}
    {Eppstein}, D. (2018).
    \newblock \emph{Edge bounds and degeneracy of triangle-free penny graphs and
      squaregraphs}.
    \newblock Journal of Graph Algorithms and Applications \textbf{22}, 483--499.
    
    \bibitem{harborth}
    Harborth, H. (1974).
    \newblock \emph{Problem 664{A}}.
    \newblock Elemente der Mathematik \textbf{29}, 14--15.
    
    \bibitem{oberwolfach}
    Harborth, H. (1981).
    \newblock \emph{Point sets with equal numbers of unit-distant neighbors} (Abstract),
    \newblock Discrete Geometry, 12--18 July 1981, Oberwolfach,
    \newblock Tagungsbericht 31/1981, Mathematisches Forschungsinstitut Oberwolfach. pp.~11--12.

    \bibitem{lighter}
    Harborth, H. (1986).
    \newblock \emph{Match sticks in the plane}.
    \newblock In: The Lighter Side of Mathematics,
      edited by R.~K. Guy and R.~E. Woodrow, 281--288. Mathematical Association of
      America, Washington, D.C.
    
    \bibitem{kurz}
    Kurz, S. and R.~Pinchasi (2011).
    \newblock \emph{Regular matchstick graphs}.
    \newblock Amer.\ Math.\ Monthly \textbf{118}, 264--267.
    
\end{thebibliography}
\end{document}